\documentclass{amsart}

\usepackage{epsfig}

\usepackage{amssymb}
\usepackage{amsmath}

\newtheorem{thm}{Theorem}[section]  
\newtheorem{lem}[thm]{Lemma}	       
\newtheorem{crlr}[thm]{Corollary}      
     
\newtheorem{defin}[thm]{Definition}%[section]
\newtheorem{rem}[thm]{Remark}%[section] 

\makeatletter
\def\theequation{\thesection.\@arabic\c@equation}
\makeatother

\def\Xint#1{\mathchoice
    {\XXint\displaystyle\textstyle{#1}}%
     {\XXint\textstyle\scriptstyle{#1}}%
     {\XXint\scriptstyle\scriptscriptstyle{#1}}%
     {\XXint\scriptstyle\scriptscriptstyle{#1}}%
	\!\int}
\def\XXint#1#2#3{{\setbox0=\hbox{$#1{#2#3}{\int}$}
	\vcenter{\hbox{$#2#3$}}\kern-.5\wd0}}

\begin{document}

\title[Weighted $L^p$ estimates for elliptic equations]{Weighted $\boldsymbol{L^p}$-estimates for elliptic equations with measurable coefficients in nonsmooth domains}

\author[S.-S. Byun, D.K. Palagachev]{Sun-Sig Byun \and Dian K. Palagachev}

\address{Sun-Sig Byun: Department of Mathematics and
               Research Institute of Mathematics,
               Seoul National University, Seoul 151-747, Korea}
\email{byun@snu.ac.kr} 
 
\address{Dian K. Palagachev: Dipartimento di Matematica, Politecnico di Bari, 70 125 Bari, Italy}
\email{palaga@poliba.it; dian@dm.uniba.it}



\begin{abstract}
We obtain a global weighted $L^p$ estimate for the gradient of the weak solutions to
divergence form elliptic equations with measurable coefficients in a nonsmooth bounded domain. The coefficients are assumed to be merely measurable in one variable and to have small BMO semi-norms in the remaining variables, while the boundary of the domain is supposed to be Reifenberg flat, which goes beyond the category of domains with Lipschitz continuous boundaries. As consequence of the main result, we derive global gradient estimate for the weak solution in the framework of the Morrey spaces which implies global  H\"older continuity of the solution.
\end{abstract}

\maketitle

\section{Introduction}\label{sec1}
\setcounter{equation}{0}
\setcounter{thm}{0}

This work is concerned with weighted $L^p$-regularity of the gradient of weak solutions to the Dirichlet problems regarding elliptic equations with possibly measurable coefficients in nonsmooth domains.
The problems in mind are related to some important variational problems arising in the mechanics of membranes and films of simple nonhomogeneous materials which form a linear laminated medium. In particular, a highly twinned elastic or ferroelectric crystal is a situation where a laminate appears. The equilibrium equations of such linear laminates usually have merely bounded measurable coefficients, see \cite{ATSP1,CKV1,ERS1,FFC1,LN1}.

In this paper we consider a nonhomogeneous elliptic equation in divergence form with bounded measurable coefficients in a very nonsmooth domain beyond the class of domains with Lipschitz continuous boundaries. Precisely, we deal with the Dirichlet problem
\begin{equation}\label{000}
\begin{cases}
D_i \left( a^{ij}(x) D_j u\right)  =  &\!\!\!  D_i f^i(x)  \quad  \textrm{in}\   \Omega, \\
      \hfill   u    =  & \!\!\! 0 \qquad\qquad    \textrm{on} \
\partial  \Omega.
\end{cases}
\end{equation}

Throughout the paper the standard summation notation is employed for $ 1 \leq i, j \leq n$ with $n \geq 2.$ Here $\Omega$ is a bounded open domain in $\mathbb{R}^n$ with  boundary $\partial \Omega$ and $F(x)=\left(f^{1}(x),\ldots,f^n(x)\right)$ is a vector valued function belonging to a suitable weighted $L^{p}$ space. The matrix of the coefficients
$A(x)=\left\{a^{ij}(x)\right\}\colon\ \mathbb{R}^{n} \rightarrow \mathbb{R}^{ n^2 }$ is assumed to be uniformly bounded and uniformly elliptic. Namely, we suppose that there exist positive constants $L$ and $\nu$ such that
\begin{equation}\label{001}
\left\Vert A\right\Vert_{ L^{\infty}(\mathbb{R}^n , \mathbb{R}^{n^2})} \leq L,
\end{equation}
and
\begin{equation}\label{002}
a^{ij}(x)  \xi_i \xi_j \geq \nu |\xi|^2
\end{equation}
for all vectors $ \xi\in \mathbb{R}^{n}$ and for almost every $x \in
\mathbb{R}^{n}.$

It is well known that under these basic assumptions imposed on $A$ and $\Omega,$ the Dirichlet problem \eqref{000} has a unique weak solution if $|F|^2 \in L^1(\Omega).$ That is, the zero extension $\bar{u}$ of $u$ belongs to $H^1_0(\mathbb{R}^n)$ and satisfies the corresponding weak integral formulation
\begin{equation}\label{def003}
\int _{\Omega }  a^{ij} D_{j} u D_{i}  \varphi \;  dx  =
\int _{\Omega } f^{i} D_{i}  \varphi  \;  dx
\end{equation}
for all $\varphi \in H^{1}_{0}(\Omega).$ Moreover, the standard $L^2$-estimate
\begin{equation}\label{004}
\int _{\Omega } |D u|^2  \;  dx
\leq c \int _{\Omega } |F|^2 \;  dx
\end{equation}
holds true for the gradient $Du$ of the weak solution with a positive constant $c$ depending on
$n,$ $\nu,$ $L$ and the Lebesgue measure $|\Omega|$ of the domain $\Omega.$

A natural extension of the $L^2$-estimate \eqref{004} is the following $L^p$-estimate
\begin{equation}\label{005}
\int _{\Omega } |D u|^p  \;  dx
\leq c \int _{\Omega } |F|^p \;  dx
\end{equation}
with exponent $1<p<\infty$
for some $c=c(n, \nu, L, p, |\Omega|)>0.$
Needless to say, for this estimate to be valid, the basic structure requirements \eqref{001} and \eqref{002} on $A$ and the boundedness of $\Omega$ are generally not enough. Some additional regularity condition on $A$ and some finer geometric assumption on $\partial \Omega$ must be imposed. A classical problem is to find such a minimal conditions on $A$ and  $\partial \Omega$ under which the estimate \eqref{005} is true for all $p$ in the range $(1,\infty).$ This is the so-called optimal $W^{1,p}$-regularity (or equivalently, maximal regularity) problem regarding \eqref{000}. As far as such minimal conditions on $A$ and $\Omega$ for the $W^{1,p}$-regularity, we are dealing here  with coefficients matrix $A$ having entries of small bounded mean oscillation (BMO) with respect to some of the variables and Reifenberg flat domains $\Omega,$ respectively. We refer to \cite{JN1,Sa1,R1,T1} for the definitions and the basic properties of the BMO space and the Reifenberg flat domains (see also \cite{BW1,BW2,P0,P1,PS}).

There have been many research activities on the $W^{1,p}$-regularity problem, cf. \cite{AM1,AM2,AQ1,BW1,CP1,Di1,Mo1,P1,Si1} for instance, most of these considering principal coefficients of the elliptic operator belonging to the spaces of functions with vanishing mean oscillation (VMO) or small BMO. However, little is known about that question in the case of \textit{only} measurable coefficients $a^{ij}.$ There is a classical by now example due to Meyers \cite{Me1}, showing that if these are \textit{merely} measurable with respect to \textit{two} independent variables, then the $W^{1,p}$-regularity of \eqref{000} fails in general. Indeed, the Meyers counterexample is easily extendable to the $n$-dimensional case $(n>2)$ of equations with coefficients that are \textit{only} measurable with respect to two of the variables.

According to the recent works in \cite{BRW1,BW2} one can allow the
coefficients $a^{ij}$ to be merely measurable in \textit{one} of the
variables for the estimate \eqref{005} to be true, while the
boundary of the domain $\Omega$ belongs to the class of Reifenberg
flat domains which goes beyond the category of sets with Lipschitz
continuous boundaries. More precisely, the global $L^p$-estimate
holds true for all $p\in(1,\infty)$ if for each point and for each
scale the coefficients are \textit{only} measurable in \textit{one}
variable and are averaged in the sense of small BMO with respect to
the remaining $n-1$ variables, while the boundary $\partial\Omega$
can be trapped into two hyperplanes depending on the scale chosen.
Let us emphasize on the fact that there is no any regularity
assumption with respect to \textit{one}  of the variables and the
boundary of the domain, being Reifenberg flat, can have rough enough
fractal structure (see the excellent survey by Toro~\cite{T1}). 

Very recently, Dong, Kim and Krylov in \cite{D1, DK1, Kr2} obtained a global $L^p$-estimate under similar assumptions as these in \cite{BRW1, BW2}. The approach used there is based on making use of a~priori pointwise estimates and working on the sharp maximal function of the gradient of solutions. This interesting technique is first developed by Krylov in \cite{Kr} and gives a unified approach for both divergence and nondivergence form equations.

The present article is a natural
outgrowth of \cite{BW2} and deals with weighted
$W^{1,p}$-theory for the Dirichlet problem \eqref{000}. In
particular, we derive an extended version of the $L^p$-estimate
\eqref{005} in the settings of the weighted Lebesgue spaces
$L^p_w(\Omega),$ generalizing this way the $W^{1,p}$-regularity
theory recently elaborated in \cite{BRW1,BW2}.
More precisely, we prove that under the same regularity assumptions on $A$ and $\Omega$ as these in \cite{BRW1,BW2},  the following global weighted $W^{1,p}$-regularity
\begin{equation}\label{006}
|F|^2 \in L^{\frac{p}{2}}_w(\Omega) \
    \Longrightarrow \   |D u|^2 \in L^{\frac{p}{2}}_w (\Omega)
\end{equation}
holds
for any $p \in (2,\infty)$ and with a weight $w$ belonging to the Muckenhoupt class $A_{\frac{p}{2}}$ (we refer the reader to Section~\ref{sec2} for the precise definitions and notations). Let us point out that the $W^{1,p}$-regularity derived in \cite{BRW1,BW2} is a special case of \eqref{006} when $w(x)\equiv 1,$ which makes \eqref{006}  a natural extension of the  $W^{1,p}$-theory.

Throughout the paper, the case in mind is when $2 < p <\infty$ and
the prescribed assumption on the free term in \eqref{000} is
$$
|F|^2 \in L^{\frac{p}{2}}_w( \Omega ), \  w \in A_{\frac{p}{2}},\quad 2<p<\infty.
$$
Under this condition, needless to say, the problem \eqref{000} has a unique weak solution $u \in H^1_0(\Omega)$ because $L^{\frac{p}{2}}_w( \Omega ) \subset L^1(\Omega).$ Assuming
that at each point and for each scale the coefficients of \eqref{000} are \textit{only} measurable in \textit{one} variable and are averaged in the sense of small BMO with respect to the remaining $n-1$ variables, and considering domains $\Omega$ with Reifenberg flat boundaries, we will prove that
$$
|Du|^2 \in L^{\frac{p}{2}}_w( \Omega )
$$
with the corresponding gradient estimate
\begin{equation}\label{017}
\Vert |D u|^2 \Vert_{L^{\frac{p}{2}}_w(\Omega)} \leq c \Vert |F|^2 \Vert_{L^{\frac{p}{2}}_w(\Omega)}
\end{equation}
for every $p \in (2, \infty),$ where the constant $c$ is independent of $u$ and $F.$

Similar weighted results have been recently published in \cite{MP1} under the more restrictive assumption on BMO smallness with respect to \textit{all independent variables.}

Our approach in this paper is based on the method of approximation which was developed by Caffarelli and Peral in \cite{CP1} in the context of maximal function technique, and later adopted for operators with discontinuous coefficients and nonsmooth
domains in \cite{BW1,BW2}. We will use maximal functions, the Vitali type covering lemma and
scaling arguments in $L^2$-estimate for the gradient of the weak solution to \eqref{000}, in order to derive
suitable decay estimates for the level sets of the Hardy-Littlewood
maximal function for $|Du|^2$ to increasing levels. After that, the standard procedure of weighted
integration over the level sets gives the desired estimate \eqref{017}. A key point of this approach is comparison with the solutions to a limiting problem, obtained from \eqref{000} by averaging in the variables with respect to which the coefficients have small BMO norms. The new coefficients depend on the remaining \textit{one} variable, and even if these are \textit{only} measurable, it turns out that the limiting problem supports $W^{1,\infty}$-regularity.

We would like to point out that it is also possible to use a very influential method by Acerbi and Mingione in \cite{AM1,AM2,Mi1}.

The paper is organized as follows. In Section~\ref{sec2} we introduce the regularity assumptions on the coefficients $a^{ij}$ and the boundary of the domain $\Omega$ in order to state the main result Theorem~\ref{thm104}. In Section~\ref{sec3} we establish local Lipschitz regularity for the solutions to the limiting problem mentioned above. Section~\ref{sec4} deals with deriving of $L^2$-gradient estimates from an appropriate perturbation theory in analysis. In Section~\ref{sec5} we obtain the optimal gradient estimate \eqref{017} in weighted Lebesgue spaces regarding the Dirichlet problem \eqref{000}. For a particular choice of the weight $w$ and as an outgrowth of our main result, we obtain in Section~\ref{sec6} gradient estimates in the framework of the Morrey spaces which imply global H\"older continuity
of the weak solutions to \eqref{000}, generalizing this way the celebrated results by De~Giorgi~\cite{DG} and Morrey~\cite{Mo0}.

\section{Main result}\label{sec2}
\setcounter{equation}{0}
\setcounter{thm}{0}

We start with the following notations:
\begin{enumerate}
\item The open ball in $\mathbb{R}^{n-1}$ with center $y^{\prime}=(y_1, \cdots,y_{n-1})$ and radius $r>0$
is denoted by $$B^{\prime}_{r}(y^{\prime})=\{x^{\prime}=(x_1, \cdots,x_{n-1}) \in
\mathbb{R}^{n-1}\colon\ |x^{\prime}-y^{\prime}| < r \}.$$
\item The cylinder in $ \mathbb{R}^{n-1} \times \mathbb{R}$ with center $y=
(y^\prime,y_n)$ and size $r>0$ in the $x_n$-axis is denoted
by
$$
C_{r}(y)= B^{\prime}_{r}(y^{\prime}) \times (y_n-r,y_n+r).
$$
If the center is the origin,  we do not specify it and write just $C_r$ for the sake of simplicity.
\item For each fixed $x_{n} \in \mathbb{R}$ and for each bounded subset
$U^\prime$ of $\mathbb{R}^{n-1},$ the integral average of a function $g(\cdot,x_n)$ with respect to $x^\prime$-variables in $U^\prime$ is denoted by
$$\overline{g}_{U^\prime}(x_n) = \Xint-_{U^\prime}
g(x^{\prime},x_n) \;  dx^{\prime}   =  \frac{1}{ |U^\prime| } \int_{ U^\prime }
g(x^{\prime},x_n)   \;  dx^{\prime},
$$
and $|U^\prime|$ stands for the $(n-1)$-dimensional Lebesgue measure of $U^\prime.$
\end{enumerate}

We now state the main assumptions on the data of problem \eqref{000} regarding the coefficients matrix $A(x)$ and the domain $\Omega.$
\begin{defin}\label{def100}
We say that $(A,  \Omega)$ is $(\delta, R)$-\textit{vanishing of codimension 1\/}
if for every point $y \in \Omega$ and for every number $r \in
\left(0, R \right]$ such that
$$
\textrm{dist}\,(y,
\partial \Omega)=\min_{ x \in
\partial \Omega } \textrm{dist}\,(y, x) >  \sqrt{2} \ r,
$$
there exists a coordinate system depending on $y$ and $r,$
whose variables we still denote by $x=(x^{\prime},x_n),$ such that in this new coordinate system $y$ is the origin and
\begin{equation}\label{101}
 \Xint-_{C_{\sqrt{2}r}} \left |A(x^\prime,x_n) -
\overline {  A } _{ B^{\prime  }_{\sqrt{2} r }} (x_n) \right| ^{2} \; dx
\leq \delta^2,
\end{equation}
while, for every point $y \in \Omega$ and for every number $r \in
\left(0, R \right]$ with
$$
\textrm{dist}\,(y,\partial \Omega)=
\min_{ x \in\partial \Omega }\textrm{dist}\,(y,x)=
\textrm{dist}\,(y,x_0) \leq  \sqrt{2} r $$
for some $x_0 \in \partial \Omega,$ there exists a coordinate system depending on $y$ and $r,$ whose
variables we still denote by $x=(x^{\prime},x_n),$ such that in this new coordinate system $x_0$ is the origin,
\begin{equation}\label{102}
C_{3 r} \cap \{x\colon\ x_n> 3 r \delta\}   \  \subset   \
C_{ 3 r }  \cap   \Omega  \   \subset       \ C_{3 r} \cap \{x\colon\ x_n
>-3 r \delta\}
\end{equation}
and
\begin{equation}\label{103}
\Xint-_{ C_{3r } } \left | A(x',x_n) - \overline
{ A} _{ B^{\prime}_{3r }} (x_n) \right| ^{2} \; dx  \leq \delta^2.
\end{equation}
\end{defin}

\begin{rem}\label{remnew}\em
1. By a scaling invariance property (see Lemma~\ref{lem408} below), one can take for simplicity $R=1$ or any other constants bigger than $1.$ On the other hand, $\delta$ is a small positive constant, being invariant under such a
scaling.

\medskip

2. If $(A, \Omega )$ is $(\delta, R)$-vanishing of codimension $1,$ then for each point and for each sufficiently small scale
there is a coordinate system so that the coefficients have small bounded mean oscillation (BMO) in $x^\prime$-directions with \textit{no regularity conditions\/} required with respect to the $x_n$-variable. Regarding the boundary of the domain, it is sufficiently flat in the Reifenberg sense in this new coordinate system. In other words, the codimension $1$ $(\delta, R)$-vanishing property of $(A, \Omega )$ is a general enough condition which is surely satisfied in the particular cases of continuous or VMO coefficients $a^{ij}(x)$ and $C^1$ or Lipschitz continuous boundary $\partial\Omega$ with small Lipschitz constant (cf. \cite{P0,P1}). It is clear that assumptions
\eqref{101} and \eqref{103} allow quite arbitrary discontinuities of $a^{ij}(x)$ in \textit{one} direction, whereas the discontinuities with respect to the remaining variables are controlled in terms of small BMO (think, for example, for small multipliers of the Heaviside step function). Moreover, the Reifenberg flatness \eqref{102} extends the $W^{1,p}$-regularity of \eqref{000} to the case of domains with rough boundaries of fractal nature (cf. \cite{T1}).

\medskip

3. The Reifenberg flatness condition \eqref{102} implies that the boundary $\partial\Omega$ satisfies the so-called (A)-condition (see \cite{LU,Cm}). Namely, setting $B_r(x_0)$ for the ball of radius $r$ and centered at $x_0,$ there exists a positive constant $K_\Omega(\delta)$ such that the Lebesgue measure of $B_r(x_0)\cap\Omega$ is comparable to that of $B_r(x_0):$
\begin{equation}
\label{mds} K_\Omega(\delta) |B_r(x_0)|\leq |B_r(x_0)\cap\Omega|\leq
\big(1-K_\Omega(\delta)\big) |B_r(x_0)|\quad \text{for each}\
x_0\in\partial\Omega.
\end{equation}

\medskip

4. The numbers $\sqrt{2}r$ and $3r$ above are selected for our purpose. The reason for this selection is that we need to take the size of a cylinder $C_r(y)$
large enough to contain its rotations  in any directions.
\end{rem}

\medskip

Before stating our main result, let us recall the definition of the Muckenhoupt classes
$A_s,$ $1<s<\infty,$ and the respective weighted Lebesgue spaces $L^{s}_w (\Omega).$
A positive locally integrable function $w$ on $\mathbb{R}^n,$ $w \in L^1_{loc}(\mathbb{R}^n),$ is called to be a weight. Then, given $s\in(1,\infty),$ this weight belongs to the Muckenhoupt class $A_{s}$ if
$$
[w]_s= \sup_{y=(y^\prime,y_n) \in \mathbb{R}^n} \sup_{r >0}  \left( \Xint-_{C_r(y)}  w(x) \;  dx \right) \left(\Xint-_{C_r(y)} w(x)^{\frac{-1}{s-1}} \;  dx \right)^{s-1}  <  \infty,
$$
where $\int \hspace{ -0.3cm } -$ is the integral average and the supremum is taken over all cylinders
$$
C_r(y)=\{x=(x^\prime,x_n) \in \mathbb{R}^n\colon\ |x^\prime-y^\prime|<r, |x_n -y_n|<r\}.
$$
We note that the $A_s$ classes are nested, that is, $A_{s_1} \subset A_{s_2}$ if $1 < s_1 \leq  s_2 < \infty.$ To give an example, consider the function
$$
w_\alpha(x) = |x|^\alpha,  \quad x \in \mathbb{R}^{n}.
$$
Then $w_\alpha \in A_{s}$ if and only if $-n < \alpha < n(s-1).$ Thus, $w_\alpha$ is a typical weight which can be considered in the present paper.

For each measurable set $E\subset \mathbb{R}^n $ and a weight $w,$ we set
$$
w(E) = \int_{E} w(x)\; dx.
$$

In what follows, we will use the following important properties of the $A_s$ weights.
\begin{lem}\label{lem010}
{\em (\cite{To1})} Let $w\in A_s$ for some $1< s <\infty,$ and let
$C_r(y)$ be the cylinder $C_r(y)$ centered at $y=(y^\prime,y_n) \in
\Omega$ and of size $r>0.$ Then we have
$$
\frac{1}{\gamma_1} \left[\frac{|C_r(y) \cap \Omega|}{|C_r(y)|}\right]^s  \leq     \frac{w\left(C_r(y)\right)} {w\left(C_r(y) \cap \Omega\right)}
 \leq \gamma_1 \left[\frac{|C_r(y) \cap \Omega|}{|C_r(y)|}\right]^\beta,
$$
where $\gamma_1$ and $\beta > 0$ are constants depending only on $[w]_s$ and $n.$
\end{lem}

We next introduce the weighted Lebesgue spaces under consideration in this paper.
Given a weight $w \in A_s,$
$1<s<\infty,$ the weighted Lebesgue space $L^s_w(\Omega)$ is the
set of all measurable functions $h\colon\ \Omega \to \mathbb{R}$
satisfying
$$
\left\|h \right\|_{L^s_w(\Omega)} =\left( \int_{\Omega} |h(x)|^s  \ w(x) \; dx \right)^{\frac{1}{s}} < \infty.
$$

\medskip

We are in a position now to state the main result of the paper.
\begin{thm}\label{thm104}
Given a number $p \in (2, \infty)$ and a weight $w \in A_{\frac{p}{2}},$ there exist a small
positive constant $\delta=\delta(\nu, L, n, p,$ $[w]_{\frac{p}{2}}, \Omega)$ and a positive constant $c(\nu, L, n, p, [w]_{\frac{p}{2}}, \Omega)$ such that if $(A, \Omega)$ is
($\delta,R)$-vanishing of codimension $1$ and $|F|^2 \in L^{\frac{p}{2}}_{w}(\Omega),$ then the unique
weak solution $u \in H^1_0(\Omega)$ of \eqref{000} satisfies $|Du|^2 \in L^{\frac{p}{2}}_{w}(\Omega)$ with the estimate
\begin{equation}\label{105}
\int_{\Omega} |Du|^{p} w(x) \; dx \leq c \int_{\Omega} |F|^{p} w(x) \; dx.
\end{equation}
\end{thm}

Let us point out that the approach employed in the paper is applicable both to equations
including lower-order terms and to elliptic systems.
In order to fix the ideas and to avoid unessential technicalities, we limit ourselves to the equations of principal type as the one considered in \eqref{000}.

\section{Lipschitz regularity for a limiting problem}\label{sec3}
\setcounter{equation}{0}
\setcounter{thm}{0}

In this section we will prove local Lipschitz regularity for a limiting problem which guarantees a fundamental step to derive \eqref{105}. For,  consider the following elliptic equation with  coefficients depending on one spatial variable, say $x_n:$
\begin{equation}\label{200}
D_i \left( \overline{a}^{ij}(x_n) D_{j}v\right)  =  0  \quad   \textrm{in}\quad
C_2 = \{(x^\prime,x_n)\colon\ |x^\prime|<2, |x_n|<2\}.
\end{equation}
Of course, the coefficients $\overline{a}^{ij}$ are assumed to satisfy the basic structure conditions \eqref{001} and \eqref{002}, but these are allowed to be \textit{only} measurable. Indeed, the solutions under consideration are defined in the weak sense as usual.
Namely, we say that $v \in H^1(C_2)$ is a weak solution of \eqref{200} if
$$
\int _{C_2 }  \overline{a}^{ij} D_{j} v  D_{i} \varphi \;  dx  =0
$$
for all $\varphi \in  H^1_{0}(C_2).$

Throughout this section we
denote by $c$ a positive constant that can be computed in terms of known quantities such as $\nu,$ $L$ and $n.$ The standard local $L^2$-estimate for \eqref{200} is
$$
\Vert D v \Vert_{L^2(C_1)} \leq c  \Vert v \Vert_{L^2(C_2)}.
$$

Since the coefficients in \eqref{200} depend \textit{only} on $x_n$ and the equation is linear, one is allowed to differentiate \eqref{200} up to any order with respect to $x^\prime=(x_1, \cdots,x_{n-1})$ variables. This observation gives the following lemma.
\begin{lem}\label{lem203}
If $v$ is  a weak solution of \eqref{200}, then so are $D^\prime v= D_{x^\prime }  v$ with the estimate
$$
\Vert D D^\prime v \Vert_{L^2(C_1)} \leq c  \Vert v \Vert_{L^2(C_2)}.
$$
\end{lem}

Moreover, we have the following improving-of-regularity result.
\begin{lem}\label{lem204}
If $v$ is a weak solution of \eqref{200}, then $D v$ belongs to $L^{2^{*}}(C_1)$ with the estimate
$$
\Vert  D v \Vert_{L^{2^{*}}(C_1)}  \leq c   \Vert v \Vert_{L^2(C_2)},
$$
where $2^{*}$ is the Sobolev conjugate of $2,$ that is, $2^*=\frac{2n}{n-2}$ if $n>2$ while $2^*$ is an arbitrary large number if $n=2.$
\end{lem}
\begin{proof}
Let $v$ be a weak solution of \eqref{200}. Then by Lemma~\ref{lem203}, $D D^\prime v  \in  L^2(C_2)$ with the estimate
$$
\Vert  D D^\prime  v \Vert_{L^2(C_1)} \leq c  \Vert v \Vert_{L^2(C_2)},
$$
whence, the Gagliardo--Nirenberg--Sobolev inequality implies that
\begin{equation}\label{206}
\Vert D^\prime v  \Vert_{L^{2^*}(C_1)} \leq c \Vert v \Vert_{L^2(C_2)}.
\end{equation}

We next want to show that $D_{n} v$ belongs to $L^{2^*}(C_1).$ To do this, remember first of all  that the coefficients $\overline{a} ^{ij}$ depend \textit{only} on $x_n$
and this rewrites the equation \eqref{200} in the form
\begin{align}\label{207}
D_{n} \left(  \overline{a}^{nn} D_{n}v   +  \sum_{ j\neq n }  \overline{a}^{nj}D_{j}v \right)  =\ & -\sum_{i\neq n}  {D_{i}} (\overline{a} ^{ij }D_{j} v) \\
\nonumber
   = \  &  -\sum_{i\neq n}  \overline{a} ^{ij }D_{ij} v \quad \textrm{in}\ C_2
\end{align}
in weak sense. Defining
\begin{equation}\label{208}
\sigma_n(v)= \overline{a}^{nn} D_{n}v   +  \sum_{ j\neq n }  \overline{a}^{nj}D_{j}v,
\end{equation}
it follows from \eqref{001}, Lemma~\ref{lem203} and \eqref{208} that
$$
D_i (\sigma_n(v))=\overline{a}^{nn} D_{n_i}v   +  \sum_{ j\neq n }  \overline{a}^{nj}D_{ij}v \in L^2(C_1)
$$
for each $i\neq n$ with the estimate
\begin{equation}\label{210}
\Vert D_i( \sigma_n(v))\Vert_{L^2(C_1)} \leq c \Vert v\Vert_{L^2(C_2)}.
\end{equation}
On the other hand, \eqref{001}, Lemma~\ref{lem203} and \eqref{207} yield
$$
D_n (\sigma_n(v))=-\sum_{i\neq n}  \overline{a} ^{ij }D_{ij} v \in L^2(C_1)
$$
and
\begin{equation}\label{212}
\Vert D_n( \sigma_n(v))\Vert_{L^2(C_1)} \leq c \Vert v\Vert_{L^2(C_2)}.
\end{equation}
It follows from \eqref{210} and \eqref{212} that
$$
\sigma_n(v) \in H^1(C_1)
$$
with the estimate
$$
\Vert \sigma_n(v)  \Vert_{H^1(C_1)} \leq c  \Vert v \Vert_{L^2(C_2)},
$$
and applying once again the Gagliardo--Nirenberg--Sobolev inequality we get
\begin{equation}\label{215}
\sigma_n(v)  \in L^{2^*}(C_1)
\end{equation}
with
$$
\Vert \sigma_n(v)  \Vert_{L^{2^*}(C_1)}  \leq c   \Vert v \Vert_{H^1(C_2)}  \leq  c   \Vert v \Vert_{L^2(C_2)}.
$$
On the other hand,  we have
$$
\overline{a}^{nn}D_{n}v = -\sum_{ j\neq n }  \overline{a}^{nj} D_{j}v +  \sigma^i_n(v)
$$
by \eqref{208} and \eqref{001}, \eqref{206} and \eqref{215} give
$$
\overline{a}^{nn}D_{n} v \in L^{2^*} (C_1).
$$

We use the basic structure conditions \eqref{001} and \eqref{002} to obtain that
$ D_{ n } v \in L^{2^*}( C_1 )$ with
\begin{equation}\label{219}
\Vert D_{n} v \Vert_{L^{2^*}(C_1)} \leq c   \Vert v \Vert_{L^2(C_2)}.
\end{equation}
The claim follows from \eqref{206} and \eqref{219} and this completes the proof.
\end{proof}

We will prove now interior $W^{1,\infty}$-regularity for the solution of \eqref{200} by employing an iteration argument.
\begin{lem}\label{lem220}
If $v$ is a weak solution of \eqref{200}, then $v \in W^{1,\infty}(C_1)$ and there exists a constant $c,$ independent of $v,$ such that
\begin{equation}\label{220}
\Vert  D v \Vert_{L^{\infty}(C_1)} \leq c  \Vert v \Vert_{L^2(C_2)}.
\end{equation}
\end{lem}
\begin{proof}
Let $v$ be a weak solution of \eqref{200}. Then it follows from Lemma~\ref{lem203} that $D^\prime v$ are also weak solutions of \eqref{200}.
According to Lemma~\ref{lem204}, $D D^\prime v \in L^{ 2^{*} }(C_1)$ with the estimate
$$
\Vert  D D^\prime v \Vert_{L^{2^{*}}(C_1)} \leq c  \Vert v \Vert_{L^2(C_2)}.
$$
The same arguments as in the proof of Lemma~\ref{lem204}, applied to
 $2^*$ instead of $2,$ give
$D v  \in W^{1,   2^{*} }(C_1)$ and
$$
\Vert  D  v \Vert_{W^{1,  2^{*} }(C_1)} \leq c  \Vert v \Vert_{L^2(C_2)}.
$$

To proceed further, we consider first the case of an \textit{odd} dimension $n.$
Repeating $k$ times the above procedure gives $2^{ {\overbrace{* \cdots *}}^{k \ \textrm{times}}}= \frac{2n}{n-2k}.$ Taking $k$ so large that $ k > \frac{n-2}{2}$
ensures $ 2^{ {\overbrace{* \cdots *}}^{k \ \textrm{times}}}>n$ and
we apply the Morrey inequality to find that
$$
D v \in W^{1, 2^{  {\overbrace{* \cdots *}}^{ k \ \textrm{times}}}}\left(C_1\right)   \subset  C^{0, \gamma_1}(C_1)  \subset L^\infty(Q_1)
$$
for some $0< \gamma_1 < 1.$ This implies
the estimate \eqref{220} when $n$ is odd and we are done.

Alternatively, if $n$ is \textit{even}, it may happen that
$$
2^{ {\overbrace{* \cdots *}}^{k \ \textrm{times}} } = \frac{2n}{n-2k}=n
$$
after $k$ iterations. However, using the quantity $\frac{2n}{n-2(k-1)}-\varepsilon$ instead of  $\frac{2n}{n-2(k-1)}$ for sufficiently small $\varepsilon >0,$ we will have
$$
\left(\frac{2n}{n-2(k-1)}-\varepsilon\right)^*  < n, \ \left(\frac{2n}{n-2(k-1)}-\varepsilon\right)^{**} >n,
$$
and the proof completes as above.
\end{proof}

Our next step consists in proving $W^{1, \infty}$-regularity up to the boundary
for the weak solutions to the Dirichlet problem for the equation \eqref{200}.
To do this, we define
\begin{equation}\label{223}
C^{+}_2=C_2 \cap \{x_n >0\}\quad \textrm{and}\quad T_2= B^\prime_2 = C_2 \cap \{x_n=0\}
\end{equation}
and suppose that $v \in H^{1}(C^+_2)$ is a weak solution of
\begin{equation}\label{224}
\begin{cases}
D_{i}(\overline{a}^{ij}(x_n)  D_{j} v)  =  &\!\!\! 0
\quad  \textrm{in} \  C^+_2, \\
\hfill v    =  &\!\!\!  0   \quad  \textrm{on}\  T_2.
\end{cases}
\end{equation}
This means that for all $\varphi \in  H^1_{0}(C^+_2),$
$$
\int _{ C^+_2 }  \overline{a}^{ij} D_{j} v \varphi^i \;  dx  =0
$$
and the zero extension of $v$ belongs ro $H^1(C_2).$

We are ready now  to give a natural extension of the interior
$W^{1,\infty}$-regularity in Lemma~\ref{lem220} up to the flat boundary using a proper reflection argument.
\begin{lem}
\label{lem226}
If $v$ is a weak solution of \eqref{224}, then $v \in W^{1,\infty}(C^+_1)$ and it satisfies the estimate
$$
\Vert  D v \Vert_{L^{\infty}(C^+_1)} \leq c  \Vert Dv \Vert_{L^2(C^+_2)}
$$
with a constant independent of $v.$
\end{lem}
\begin{proof}
Define $\overline{v}$ in $C_2$ by
$$
\overline{v}(x^\prime,x_n)=
\begin{cases}
\hfill v(x^\prime,x_n) & \textrm{ if }\ x_n \geq 0, \\
-v(x^\prime,-x_n) & \textrm{ if }\ x_n < 0
\end{cases}
$$
and extend $\overline{a}^{ij}(x_n)$ from $\{x_n >0\}$ to $\{x_n \leq 0\}$ by  even  or odd reflection, depending on the indices $i$ and $j,$ in a way that the extended $\overline{v}$ is a weak solution  of \eqref{200}.
We then apply Lemma~\ref{lem220} to find that $\overline{v}$ supports the interior $W^{1,\infty}$-regularity in $C_2.$ The claim follows from the restriction of $\overline{v}$ from $C_2$ to $C^+_2.$
\end{proof}

\section{$L^2$-estimates from an argument by perturbation}\label{sec4}
\setcounter{equation}{0}
\setcounter{thm}{0}

In this section we study gradient estimates of the weak solution of \eqref{000} by comparison with solutions to the limiting problems \eqref{200} and \eqref{224}.
The idea is to use higher integrability results, see Lemma~\ref{lem307} and Lemma~\ref{lem322}.
These regularity results follow from the fine works in the recent papers \cite{BP1,Pa1,Pa2,Pa3} by  B\"{o}gelein and Parviainen where the authors investigated self improved regularity near the boundary with a very mild condition, so-called, the capacity density condition. Needless to say, our Reifenberg flat domain satisfies this capacity density condition.

We will employ here the reverse H\"{o}lder inequality, as used in \cite{AM1,AM2,BRW1,BW2} which gives a better regularity of solutions.  This reverse H\"{o}lder inequality can compensate the lack of compactness of weak solutions which was previously used in \cite{BW1}. Thus the present approach can be applied to a more general setting when the coefficients belongs to $L^1$ with respect to one of the spatial variables where any compactness fails.

We start with the interior case. To do this, let us suppose that $F \in L^2(\Omega, \mathbb{R}^{n})$ and $u \in H^{1}(\Omega)$ satisfies
$$
\int _{\Omega }  a^{ij} D_{j} u^i \varphi \;  dx  =
\int _{\Omega } f^{i} D_{i}  \varphi  \;  dx
$$
for all $\varphi \in H^1_0(\Omega),$ that is, $u$ is a weak solution of
\begin{equation}\label{301}
D_i \left( a^{ij} D_{j} u\right)  =    D_i f^{i} \quad    \textrm{in}\   \Omega.
\end{equation}

By a proper translation, scaling and normalization (see Lemma~\ref{lem408} below) we further assume that
\begin{equation}\label{302}
C_5=\{(x^\prime,x_n)\colon\ |x^\prime|<5, |x_n|< 5\} \subset \Omega,
\end{equation}
\begin{equation}\label{303}
\Xint-_{ C_{ 4 } } |Du| ^{2} \;  dx \leq 1
\end{equation}
and
\begin{equation}\label{304}
\Xint-_{ C_{4}  } \Big(| A(x^\prime,x_n)-
        \overline{  A } _{  B^\prime _4 } (x_n)|^{2} +|F|^2 \Big) \;  dx   \leq  \delta^{2}
\end{equation}
for some small $\delta>0.$

Then we consider a local homogeneous boundary problem corresponding to \eqref{301},
\begin{equation}\label{305}
\begin{cases}
D_{i } ( a^{ij}   D_{j} \tilde{u} )  = &\!\!\! 0  \quad \textrm{in}\  C_3,\\
\hfill \tilde{u}  = &\!\!\! u \quad \textrm{on} \   \partial C_3,
\end{cases}
\end{equation}
and the limiting problem
\begin{equation}\label{306}
\begin{cases}
D_{i } \left( \overline {   a^{ij}}_{  B^{ \prime }_{ 4 } } (x_n) D_{j} v \right)  = &\!\!\! 0  \quad \textrm{in} \  C_2,\\
\hfill v  = &\!\!\! \tilde{u} \quad \textrm{on} \ \partial C_2.
\end{cases}
\end{equation}

We will use the following higher regularity result for the weak solution to \eqref{305}.
\begin{lem}\label{lem307}
If $\tilde{u}$ is the weak solution of \eqref{305} and \eqref{302}--\eqref{304} hold true, then there exists a constant $\sigma=\sigma(n,\nu,L)>0$ such that
$$\Vert D \tilde{u} \Vert_{ L^{2+\sigma}(C_2) } \leq c.$$
\end{lem}
\begin{proof}
It follows from well-known higher integrability results for \eqref{305} that there exists a constant $\sigma=\sigma(n,\nu,L)$
such that
\begin{equation}\label{308}
\Xint-_{ C_{ 2  } } |D \tilde{u}|^{2+\sigma} \;  dx \leq   c  \left(\Xint-_{ C_{ 3  } } |D \tilde{u}|^{2} \;  dx \right)^{ \frac{ 2+\sigma}{2}}.
\end{equation}
Then the maximality property\footnote{Hereafter, by ``maximality property'' we mean the standard $L^2$-estimate for the weak solutions to Dirichlet problem for second-order, divergence form linear elliptic equations.} and \eqref{303} imply that
\begin{equation}\label{309}
\Xint-_{ C_{ 3  } } |D \tilde{u}|^{2} \;  dx \leq c \Xint-_{ C_{ 3  } } |D u|^{2} \;  dx \leq c
\end{equation}
and the claim follows from \eqref{308} and \eqref{309}.
\end{proof}

\begin{lem}\label{lem310}
If $v$ is the weak solution of \eqref{306} and  \eqref{302}--\eqref{304} hold true, then $v$ is locally Lipschitz continuous with the uniform bound
$$
\Vert D v \Vert
^{2}_{L^{\infty}(C_{1})} \leq  c.
$$
\end{lem}
\begin{proof}
By Lemma~\ref{lem220}, the maximality property for  \eqref{305} and \eqref{306}, and \eqref{303} we have
$$
\Vert D v \Vert
^{2}_{L^{\infty}(C_{1})} \leq  c \Xint-_{ C_{ 2  } } |D v|^{2} \;  dx  \leq c \Xint-_{ C_{ 3  } } |D \tilde{u}|^{2} \;  dx \leq c \Xint-_{ C_{ 4 } } |D u|^{2} \;  dx \leq c.
$$
\end{proof}

\begin{lem}
\label{lem311}
Given a weak solution $u$ of \eqref{301} satisfying \eqref{302} and \eqref{303}, let $v$ be the weak solution of \eqref{306}.
Then for $0 < \varepsilon < 1 $ fixed, there exists a small $\delta=\delta(\varepsilon)>0$
such that if \eqref{304}
holds for such $\delta,$ then we have
$$
\Xint-_{C_1} |D (u-v)|^2 \;  dx \leq \varepsilon^2.
$$
\end{lem}
\begin{proof}
Let $\tilde{u}$ be the weak solution of \eqref{305}. Then it follows from a direct computation that $u-\tilde{u} \in H^1_0(C_3)$ is the weak solution of
\begin{equation}\label{313}
D_{i}( a^{ij} D_{j} (u-\tilde{u}))  =  D_i  f^i\quad  \textrm{in}\  C_3.
\end{equation}
We also see that
$\tilde{u}-v \in H^1_0(C_2)$ is the weak solution of
$$
D_{i}\left( \overline{a^{ij}}_{B^\prime_4} (x_n)   D_{j} (\tilde{u}-v)\right)  = - D_i\left( \big[a^{ij}(x^\prime,x_n)- \overline{a^{ij}}_{B^\prime_4}(x_n)\big]
D_j \tilde{u}\right)\quad   \textrm{in}\  C_2.
$$
It follows from the standard
$L^2$-estimate for \eqref{313} and the smallness condition \eqref{304} that
\begin{equation}\label{315}
\Xint-_{C_3} |D(u-\tilde{u})|^2 \;  dx \leq  c \Xint-_{C_3} |F|^2 \;  dx  \leq  c \delta^2.
\end{equation}
Furthermore, in view of H\"{o}lder's  \ inequality, Lemma~\ref{lem307}, \eqref{001} and \eqref{304} we have
\begin{align*}
 \Xint-_{C_2} |D(\tilde{u}-v)|^2  \; dx    = &  \      \Xint-_{Q_3}  \left| A -  \overline{  A } _{   B^\prime _4 } (x_n)\right|^2 |D\tilde{u}|^2 \; dx \\
 \leq\ &            c      \left(\Xint-_{Q_3} \left| A -  \overline{  A } _{  B^\prime _4 } (x_n)\right|^{    \frac{4+2 \sigma }{ \sigma }  }  \; dx\right)^{ \frac {\sigma}{2+ \sigma} }
\left(\Xint-_{C_3}  |D\tilde{u}| ^{ 2+\sigma}  \; dx\right)^{  \frac{ 2 }{ 2+ \sigma }} \\
 \leq\ &       c  \left(\Xint-_{C_4} \left| A -  \overline{  A} _{  B^\prime _4 }  (x_n) \right|^{ \frac{ 4 }{ \sigma}+2 } \; dx\right)^{ \frac{\sigma} {2+ \sigma}}\\
 \leq\ &  c \  \delta^{ \sigma_1}
\end{align*}
for some $\sigma_1=\sigma_1(n,L,\nu)>0.$
In the light of \eqref{315}, we thus deduce
$$
\Xint-_{C_1} |D(u-v)|^2 \;  dx \leq c(\delta^2 + \delta^{\sigma_1} ) =  \varepsilon^2,
$$
by taking $\delta>0$ so small, in order to get the last equality.
This completes the proof.
\end{proof}

We next extend the interior estimates from Lemma~\ref{lem311} to obtain boundary gradient estimates in Reifenberg domains. Recalling the notations from Section~\ref{sec2} and \eqref{223}, we add here some more geometric notations:

\begin{enumerate}
\item $\Omega_r = \Omega \cap C_r,$ $C^+_r = C_r \cap \{x=(x^\prime,x_n)\colon\ x_n >0\},$
\item $\partial_w \Omega_r = \partial \Omega \cap C_r,$ $T_r = C_r \cap \{x=(x^\prime,x_n)\colon\ x_n =0\}.$
\end{enumerate}

Motivated from our geometric assumption in Definition~\ref{def100}, we further assume that there exists a small $\delta>0$ such that
\begin{equation}
\label{316}
C^+_5  \subset  \Omega_5   \subset \Omega \cap \{x=(x^\prime,x_n)\colon\ x_n  > -10 \delta\},
\end{equation}
and
\begin{equation}
\label{317}
\Xint-_{ C_{4}  } \Big(\left| A(x^\prime,x_n)- \overline{  A} _{  B^\prime _4 } (x_n)  \right|^{2}  + |F|^2  \Big)\;  dx    \leq  \delta^{2}
\end{equation}
holds true for such small $\delta.$ Then let us suppose that
$u$ is a weak solution of
\begin{equation}
\label{318}
\begin{cases}
 \displaystyle  D_{i }( a^{ij} D_{j}u )  = &\!\!\! D_{i}  f^i  \quad \textrm{in} \ \Omega_5,\\
 \hfill u = &\!\!\! 0 \quad\quad\ \: \textrm{on} \ \partial_w \Omega_5,
 \end{cases}
\end{equation}
which means that for all $\varphi \in H^1_0(\Omega_5),$
$$\int_{\Omega_5} a^{ij} D_{j}u D_{i} \varphi \;  dx = \int_{\Omega_5} f^i D_{i} \varphi \;  dx$$
and the zero extension of $u$ belongs to $H^1_0(C_5).$ By means of suitable scaling and normalization, we may also assume that
\begin{equation}
\label{319}
\Xint-_{\Omega_5} |D u|^2 \  dx \leq  1.
\end{equation}
Then consider the following homogeneous problem
\begin{equation}
\label{320}
\begin{cases}
\displaystyle
D_{i }( a^{ij} D_{j} w )  = &\!\!\! 0  \quad \textrm{in} \ \Omega_5,\\
\hfill w = &\!\!\! u \quad \textrm{on} \ \partial \Omega_5.
\end{cases}
\end{equation}
We then look at the following limiting problems
\begin{equation}
\label{321}
\begin{cases}
D_{i } \left(\overline{ a_{ij}} _{  B^\prime _4 } (x_n) D_{j} h \right)
 = &\!\!\! 0  \quad \textrm{in} \ \Omega_4,\\
\hfill h  = &\!\!\! w \quad \!\textrm{on} \ \partial_w \Omega_4,
\end{cases}
\end{equation}
and
\begin{equation}
\label{322}
\begin{cases}
D_{i } \left( \overline{ a_{ij}} _{  B^\prime _4 } (x_n) D_{j} v \right)
= &\!\!\! 0  \quad \textrm{in} \ C^+_4,\\
\hfill v  = &\!\!\! 0 \quad \textrm{on} \ T_4.
\end{cases}
\end{equation}

\begin{lem}
\label{lem322}
Let $w$ be the weak solution of \eqref{320}
satisfying \eqref{316}, \eqref{317} and \eqref{319}.
Then
$$
\Vert D w\Vert_{L^{2+\sigma_2}(\Omega_4)} \leq c,
$$
for some small universal constant $\sigma_2>0.$
\end{lem}
\begin{proof}
It follows from the Reifenberg flatness condition \eqref{102} (cf. Remark~\ref{remnew}~3) that
$\partial\Omega$ satisfies the (A)-condition, that is, the Lebesgue measure of $\Omega_r(y)$ is $\delta$-comparable to $|C_r|:$
$$
|C_{r}(y)| \geq |\Omega_{r}(y)| \geq  \left[\frac{1-\delta}{2\sqrt{2}}\right]^n |C_{r}(y)|, \ \ \ \forall y \in \partial \Omega, \ \ \  \forall r >0.
$$
By using this property, one can check that $\mathbb{R}^n \setminus \Omega$ satisfies the uniform $2$-thickness condition from \cite{BP1,Pa1,Pa2,Pa3} and, as consequence, the uniform capacity density condition.  Then, according to the improving-of-integrability result up to the boundary (cf. \cite{BP1,Pa1,Pa2,Pa3}) we have
$$
\Xint-_{ \Omega_{ 4  } } |Dw|^{2+\sigma_2} \  dx \leq   c  \left(\Xint-_{ \Omega_{ 5  } } |D w|^{2} \;  dx \right)^{ \frac{ 2+\sigma_2}{2}}
\leq c  \left(\Xint-_{ \Omega_{ 5  } } |D u|^{2} \;  dx \right)^{ \frac{ 2+\sigma_2}{2}} \leq c
$$
for some positive constant $\sigma_2=\sigma_2(n,\nu,L),$ after
using the maximality property for \eqref{320} and the assumption \eqref{319}.
\end{proof}

Based on weak compactness argument, we compare the weak solution \eqref{321} with a weak solution of \eqref{322}.
\begin{lem}
\label{lem324}
For any $\varepsilon>0$, there exists a small $\delta =\delta(\varepsilon) >0$
such that if $h$ is the weak solution of \eqref{321}
with
$$ 
\Xint-_{\Omega_4} |Dh|^2 \ dx \leq 1,
$$
and \eqref{316}-\eqref{317} hold for such $\delta$,
then there exists a weak solution $v$ of (\ref{322}) such that
\begin{equation}
\label{325}
\Xint-_{C^{+}_4} |Dv|^2 \ dx \leq 1\quad \text{ and }\quad  \Xint-_{C^{+}_4}|h-v|^2 \ dx \leq \varepsilon^2.
\end{equation}
\end{lem}
\begin{proof}
If not, there would exist $\varepsilon_0>0$, $\left\{ h_k \right\}_{k=1}^{\infty}$ and
$\left\{\Omega_5^k \right\}_{k=1}^{\infty}$ such that $h_k$ is a weak solution of
\begin{equation}
\label{326}
\begin{cases}
D_{i } \left(\overline{ a_{ij}} _{  B^\prime _4 } (x_n) D_{j} h_k \right)
 = &\!\!\! 0  \quad \textrm{in} \ \Omega^k_4,\\
\hfill h_k  = &\!\!\! w \quad \textrm{on} \ \partial_w \Omega^k_4
\end{cases}
\end{equation}
with
\begin{equation}
\label{327}
C^+_5 \subset  \Omega_5^k  \subset  C_5 \cap \left\{ x_n > -\frac{10}{k}\right\},
\end{equation}
and
\begin{equation}
\label{328}
\Xint-_{\Omega_4^k} |Dh_k|^2 \ dx  \leq 1,
\end{equation}
but for any weak solution $v$ of
\begin{equation}
\label{329}
\begin{cases}
D_{i } \left(\overline{ a_{ij}} _{  B^\prime _4 } (x_n) D_{j} v \right)
= &\!\!\! 0  \quad \textrm{in} \ C^+_4,\\
\hfill v  = &\!\!\! 0 \quad \textrm{on} \ T_4,
\end{cases}
\end{equation}
with
\begin{equation}
\label{330}
\Xint-_{C_4^+} |D v|^2 \ dx  \leq 1,
\end{equation}
we have
\begin{equation}
\label{331}
\int_{C_4^+} |h_k-v|^2 \ dx  > \varepsilon_0^2.
\end{equation}

We now extend $h_k$ by zero from $\Omega^k_4$ to $C_4$ and denote it by $h_k$ also. Then it follows from Poincar\'{e}'s inequality and \eqref{328} that
$||h_k||_{H^{1}(C_4)} \leq c$. Then there exists a subsequence, which we still denote by $\left\{h_k \right\}$, and $h_{0} \in H^{1}(C_4^+)$ such that
\begin{equation}
\label{332}
h_k \rightharpoonup  h_{0} \text{ weakly in } H^{1}(C_4^+) \text{ and } h_k \rightarrow  h_{0} \text{ strongly in } L^{2}(C_4^+).
\end{equation}
From \eqref{326}, \eqref{327} and \eqref{332} we see that $h_{0}$ is a weak solution of
\begin{equation}
\label{333}
\begin{cases}
D_{i } \left( \overline{ a_{ij}} _{  B^\prime _4 } (x_n) D_{j} h_0 \right)
= &\!\!\! 0  \quad \textrm{in} \ C^+_4,\\
\hfill h_0  = &\!\!\! 0 \quad \textrm{on} \ T_4.
\end{cases}
\end{equation}
We next observe from \eqref{328}, \eqref{332} and weak lower semicontinuity property that
$$
\Xint-_{C_4^+} |Dh_0|^2 \ dx  \leq \liminf_{k \rightarrow \infty} \Xint-_{C^+_4} |Dh_k|^2 \ dx \leq 1.
$$
We then reach a contradiction to \eqref{331} from  \eqref{332}-\eqref{333}.
This completes the proof.
\end{proof}

Now we are in a position to obtain an analogue of Lemma~\ref{lem311} regarding gradient estimates up to the boundary for the weak solutions of \eqref{000}.
\begin{lem}
\label{lem334}
Let $u$ be a weak solution of \eqref{318}.
Then for any $\varepsilon > 0$, there is a small $\delta = \delta(\varepsilon) > 0 $ such that
if (\ref{316}),  (\ref{317}) and (\ref{319}) hold, then there exists a weak solution $v$ of (\ref{322}) such that
\begin{equation}\label{335}
||Dv_0||_{L^{\infty}(\Omega_3)} \leq c
\end{equation}
and
$$
\Xint-_{\Omega_2}|D(u-v_0)|^2 \leq \varepsilon^2,
$$
where $v_0$ is the zero extension of $v$ from $C_4^{+}$ to $C_4$.
\end{lem}
\begin{proof}
Let $w $ be the weak solution of \eqref{320} and $h$
the weak solution of \eqref{321}.
Then we can derive in a similar way as in the proof of Lemma \ref{lem311} that
\begin{equation}
\label{337}
\int_{\Omega_4} |D(u-h)|^2 \ dx \leq c \delta^{2}.
\end{equation}
From the maximality properties for  \eqref{320} and \eqref{321}, we obtain
\begin{equation}
\label{338}
\int_{\Omega_4} |Dh|^2 \ dx \leq c \int_{\Omega_4} |Dw|^2 \ dx \leq c  \int_{\Omega_5} |Du|^2 \ dx \leq c.
\end{equation}
In light of \eqref{338} and Lemma \ref{lem324}, we see that there is a weak solution $v$ of \eqref{322}
such that
\begin{equation}
\label{339}
\Xint-_{C_4^{+}} |Dv|^2 \ dx \leq c
\end{equation}
and
\begin{equation}
\label{340}
\int_{C_4^+} |h-v|^2 \ dx  \leq c_* \varepsilon^2,
\end{equation}
where $c_*$ is to be determined small in a universal way.

We extend $v$ by zero from $C_4^{+}$ to $C_4$ and denote it by $v_0$.
Then we derive from Lemma \ref{lem226} and \eqref{339} that
\begin{equation}
\label{341}
||Dv_0||_{L^{\infty}(\Omega_3)}=||Dv||_{L^{\infty}(C^{+}_{3})} \leq c ||Dv||_{L^{2}(C^+_4)} \leq c,
\end{equation}
which is \eqref{335}.

A direct computation shows that $v_0$ is a weak solution of
\begin{equation}
\label{342}
\left\{ \begin{array}{rclcc}
D_{i } \left( \overline{ a_{ij}} _{  B^\prime _4 } (x_n) D_{j} v_0 \right)& =  &  - D_n g^n  &  \textrm{ in }  &  \Omega_4,\\
v_0               & =  &  0   &  \textrm{ on }       &  \partial_w \Omega_4,
\end{array}\right.
\end{equation}
where
\begin{equation}
\label{343}
g^{n}=\left\{ \begin{array}{ll}
0 & \textrm{ if } x_n  \geq 0,\\
\overline{ a_{nn}} _{  B^\prime _4 } (x_n) D_n v(x^\prime,0)& \textrm{ if } x_n <0.
\end{array}
\right.
\end{equation}
Then in light of \eqref{321} and \eqref{342}, we find that
$h-v_0$ is a weak solution of
\begin{equation}
\label{344}
\left\{ \begin{array}{rclcc}
D_{i } \left( \overline{ a_{ij}} _{  B^\prime _4 } (x_n) D_{j}[h-v_0] \right)& =  &   D_n g^n  &  \textrm{ in }  &  \Omega_4,\\
h-v_0               & =  &  0   &  \textrm{ on }       &  \partial_w \Omega_4.
\end{array}\right.
\end{equation}
From standard $L^2$-estimate for \eqref{344}, we have
\begin{equation}
\label{345}
\Xint-_{\Omega_2} |D(h-v_0)|^2 \ dx \leq c \left( \int_{\Omega_3} \left(|h-v_0|^2  + |g^n|^{2}\right) dx\right).
\end{equation}
We estimate the right-hand side of \eqref{345} as follows:
\begin{align}\label{346}
\Xint-_{\Omega_3} |h-v_0|^2 \ dx    = &\ \Xint-_{ C_3^+  } | h-v |^2 \ dx + \frac{1}{|\Omega_3|}  \int_{ \Omega_3 \backslash C^+_3 } |h|^2 \ dx\\
\nonumber
  \leq &\ c_* \varepsilon^2   +  \frac{1}{|\Omega_3|} \left(\int_{ \Omega_3 } |h|^{\frac{2n}{n-2}} \ dx\right)^{\frac{n-2}{n}} \left|\Omega_3 \setminus C_3^{+}\right|^{\frac{2}{n}} \\
\nonumber
  \leq & \ c_* \varepsilon^2   +  c  \delta^{\frac{2}{n}}  \Xint-_{ \Omega_3 } |Dh|^2 \ dx\\
\nonumber
  \leq &\  c_* \varepsilon^2   +    c  \delta^{\frac{2}{n}}.
\end{align}
Here in the first line, we have used \eqref{316} and the fact that $v_0=0$ in $\Omega_4 \setminus C_4^{+}$.
In the second line, we have used \eqref{340} and H\"{o}lder's inequality. In the third line, we have used \eqref{316} and Sobolev inequality.
In the last line, we have used \eqref{338}.

Using \eqref{001}, \eqref{316}, \eqref{341} and \eqref{343}, we deduce
\begin{equation}
\label{347}
\Xint-_{\Omega_3} |g^n|^{2} \ dx  \leq    \frac{1}{|\Omega_3|} \int_{\Omega_3 \setminus C_3^{+}} \left|\overline{ a_{nn}} _{  B^\prime _4 } (x_n) D_n v(x',0)) \right|^{2} dx \leq  c \delta.
\end{equation}
But then \eqref{345}, \eqref{346} and \eqref{347} imply
\begin{equation}
\label{348}
\Xint-_{\Omega_2} |D(h-v_0)|^{2} \ dx  \leq    c  \left(c_* \varepsilon^2   +    \delta^{\frac{2}{n}} \right).
\end{equation}
We now combine \eqref{337} and \eqref{348}, to obtain
$$
\Xint-_{\Omega_2} | D (u-v_0)|^2 \ dx
\leq  c  \left( \delta +  c_* \varepsilon^2   +    \delta^{\frac{2}{n}} \right) \leq c  \left( c_* \varepsilon^2   +    \delta^{\frac{2}{n}} \right).
$$
Finally, we take $c_*$ and $\delta$ so small, in order to get the required estimate
$$
\Xint-_{\Omega_2} | D (u-v_0)|^2 \ dx
\leq  \varepsilon^2.$$
This completes the proof.
\end{proof}

\section{Gradient estimates in weighted Lebesgue spaces}\label{sec5}
\setcounter{equation}{0}
\setcounter{thm}{0}

In this section we will obtain the optimal weighted $W^{1,p}$-regularity for the Dirichlet problem \eqref{000}.
The main analytic tool of our approach is the maximal function, so let us recall first of all its definition and basic properties, see \cite{de1,St}.
\begin{defin}\label{def400}
Given a locally integrable function $h \in L^1_{loc}(\mathbb{R}^{n}),$ its Hardy--Little\-wood maximal function
$\mathcal{M}h$ is given by
$$
(\mathcal{M}h)(x)=\sup _{   C_r(x)  }\Xint-_{C_r(x) }|f(y)| \
dy,
$$
where the supremum is taken over all cylinders
$C_r(x)$ in $\mathbb{R}^{n}$ centered
at the point $x$ and of size $r>0.$ If $h$ is defined only in a
bounded domain $U \subset \mathbb{R}^{n},$ we
define its restricted maximal function as
$$
\mathcal{M}_{U}h =\mathcal{M} \left( h  \chi
_{U} \right),$$
where
$$
(h    \chi_{ U }) (x)= \begin{cases}
h(x) & \textrm{ if } x \in U, \\
0 & \textrm{ if }  x \not \in U.
\end{cases}
$$
\end{defin}

The important properties of the maximal function with respect to weights are summarized in the following
\begin{lem}
\label{lem401} {\em (\cite{To1})} Given a weight $w \in A_s$ for
some $s \in (1, \infty)$, suppose that $h\in L^s_w(\mathbb{R}^n)
\subset L^1_{loc}(\mathbb{R}^n).$ Then we have
$$
\frac{1}{c}    \Vert h\Vert _{L^{s}_w(  \mathbb{R}^{n})}
\leq    \Vert \mathcal{M} h\Vert _{L^{s}_w(  \mathbb{R}^{n})}\leq
c\Vert h \Vert _{L^{s}_w(  \mathbb{R}^{n})},
$$
where $c=c(n,s,[w]_s)>0$ is a universal constant. In the particular case $w(x)=1,$ it follows
\begin{equation}\label{403-1}
\big|\{x\in \mathbb{R}^{n}\colon\ \left(\mathcal{M} h \right)(x)
> \lambda \}\big|  \leq    c   \frac{1}{\lambda }\int |h(x)\;  dx,   \ \  (\lambda>0),
\end{equation}
with a universal constant $c=c(n)>0.$
\end{lem}

We will use also the following standard result from the  classical
measure theory regarding weighted Lebesgue spaces.
\begin{lem}\label{lem404}
{\em (\cite{To1})} Given a weight $w \in A_s$ for some $s \in (1,
\infty),$ let $h$ be a nonnegative function lying in $L^s_w(U)
\subset L^1(U)$ for some bounded domain $U \subset \mathbb{R}^n.$
Let $\theta>0$ and $\lambda >1$ be constants. Then
$$
h\in  L^{s}_w(U) \Longleftrightarrow  S = \sum _{ k \geq 1 } \lambda^{ ks  } w\big(\{ x  \in U \colon\ h(x) >  \theta  \lambda^{k}\} \big) < \infty
$$
with
$$
\frac{1}{c}S \leq     \Vert h \Vert^s_{L^{s}_w( U)}  \leq c\left(w(U) + S\right),
$$
for some universal constant $c=c(\theta, \lambda, s).$
\end{lem}

In what follows, our approach is mainly based on the following version of the Vitali
covering lemma stated in the settings of weighted measurable sets.
\begin{lem}\label{lem407}
Let $\Omega$ be a bounded,
$(\delta,1)$-Reifenberg flat set, and let $w$ be a weight from $A_s$ for some $s \in (1, \infty).$
Suppose the measurable sets $D \subset E \subset \Omega$ have the following properties:
\begin{equation}
\label{c1}
w\left(D \cap C_1(y)\right) < \varepsilon w\left( C_1(y)\right) \textrm{ for some } \varepsilon \in (0,1) \textrm{ and for every y } \in \Omega;
\end{equation}
\begin{equation}\label{c2}
\begin{cases}
  \text{for any cylinder } C_r(y) \text{ with } r>0 \text{ and } y \in \Omega,\\
w(D \cap C_r(y)) \geq \varepsilon w(C_r(y)) \text{ implies } \Omega \cap C_r(y) \subset E.
\end{cases}
\end{equation}
Then there holds
$$
w(D) \leq  \gamma_2 \varepsilon   \    w(E)
$$
for some positive constant  $\gamma_2=\gamma_2(\delta,n,s,[w]_s).$
\end{lem}
\begin{proof}
Fix any $y \in D$ and consider a continuous function
$$
\varpi(r)=\frac{ w\left(D \cap C_r(y)\right)} { w \left(C_r(y)\right)}, \ r>0.
$$
Then $\varpi(0)=\lim_{r \rightarrow 0+}\varpi(r)=1$. From the
assumption \eqref{c1}, we see $\varpi(1)< \varepsilon$. Thus, one can
find a number $r_{y}=r(y) \in (0,1)$ such that
\begin{equation}\label{v1} 
\varpi(r_{y})= \varepsilon,  \qquad \varpi(r) < \varepsilon,
\quad \forall r > r_{y}.
\end{equation}
Since $\{ D \cap C_{r_{y}}(y)\}_{y\in D }$ is an open covering of
$D$, there exist a disjoint subcovering $\{C_{r_i}(y_i)\}_{i \geq
1}$ with $r_i=r(y_i)<1$ such that
\begin{equation}\label{v2} 
\bigcup_{ i \geq 1 } [D \cap C_{5r_i}(y_i)] = D.
\end{equation}
We now compute as follows:
\begin{eqnarray*}
w(D) &  \stackrel{\eqref{v2} } {\leq }  &  w \left(\bigcup_{ i \geq
1 } [D \cap C_{5r_i}(y_i)] \right) \leq     \sum_{i \geq 1} w \left(D \cap C_{5r_i}(y_i) \right)\\
& \stackrel{\eqref{v1} } {< }  &   \varepsilon  \sum_{i \geq 1}  w
\left(C_{5r_i}(y_i)\right) \\
&  \leq  &  c   \,  \varepsilon  \sum_{i \geq 1} w \left(C_{r_i}(y_i)\right) \textrm{ by the doubling property of } w,\\
& \stackrel{\textrm{Lemma \ref{lem010}}} { \leq  }  &   c   \, \varepsilon
\sum_{i \geq
1} \left[\frac{|C_{r_i}(y_i)|}{| \Omega \cap C_{r_i }(y_i)|}\right]^{\beta}  w\left( \Omega \cap C_{ r_i }(y_i)\right)\\
& \stackrel{ \eqref{mds} } \leq &   c \left[\frac{2
\sqrt{2}}{1-\delta}\right]^{ n \beta} \varepsilon  \sum_{i \geq 1}
w\left(\Omega \cap C_{ r_i }(y_i)\right)  \\
& = & c \left[\frac{2 \sqrt{2}}{1-\delta}\right]^{ n \beta} \
\varepsilon  \ w \left( \bigcup_{i}[\Omega \cap  C_{r_i }(y_i)]\right)  \textrm{ by the disjoint covering}, \\
&  \stackrel{\eqref{c2},\eqref{v1}  } { \leq  }  &  c \left[\frac{2
\sqrt{2}}{1-\delta}\right]^{ n \beta} \ \varepsilon  \ w \left(E\right)
.
\end{eqnarray*}
This completes the proof.
\end{proof}

We will employ also the following invariance property under scaling and normalization which follows by a straightforward
 computations.
\begin{lem}\label{lem408}
Under the scaling
$$
A_\rho(x)=A(\rho x)\quad \textrm{and}\quad \Omega_\rho=\left\{\frac{1}{\rho} x \colon\  x \in \Omega \right\}
\qquad ( 0<\rho<1 ),
$$
the structure assumptions \eqref{001}, \eqref{002}
and the main assumption in Definition~\ref{def100} are invariant with the same
constants $\nu,$ $L,$ and the dilated scale $\frac{ R }{\rho}.$
Furthermore, along with the normalization
$$
u_{\lambda,\rho}(x)=\frac{u(\rho x)}{\lambda \rho }\quad \text{and}\quad
F_{\lambda,\rho}(x)=\frac{F(\rho x)}{\lambda} \ \ \ \ \ (\lambda >1),
$$
$u_{\lambda,\rho}(x)$ is a weak
solution of
$$
\begin{cases}
D_i \left(a_\rho^{ij}(x) D_{j} u_{\lambda,\rho}(x)\right)  =  &\!\!\!
D_i f_{\lambda,\rho}^{i}(x)\quad  \text{in}  \  \Omega_\rho,\ \,\,\\
   \hfill      u_{\lambda,\rho}(x)   =  &\!\!\!  0   \hfill  \text{on} \
\partial \Omega_\rho.
\end{cases}
$$
\end{lem}

Based on the maximal function and scaling argument and making use of Lemma~\ref{lem311} and Lemma~\ref{lem334}, we have the following result.
\begin{lem}\label{lem409}
Let $|F|^2 \in L^{\frac{p}{2}}_w(\Omega) \in L^1(\Omega)$ with $w \in A_\frac{p}{2},$ $2<p< \infty,$ and let $u \in H^1_0(\Omega)$ be the weak solution of \eqref{000}.
Then, there exists a universal constant $\lambda_2>1$ such that for each $0 <
\varepsilon < 1$ one can select a small $\delta=
\delta(\varepsilon)>0$ such that if  $(A,  \Omega)$ is $(\delta, R)$-vanishing of codimension $1$ and $C_r(y)$ satisfies
\begin{equation}\label{410}
w\big(\left\{x \in \Omega\colon\ \mathcal{M}(|D
u|^{2})
> \lambda_2^2\right\}   \cap  C_{r}(y) \big) \geq  \varepsilon \
w\left(C_{r}(y)\right)
\end{equation}
for such a small $\delta,$ then we have
\begin{equation}\label{411}
C_r(y) \cap \Omega  \ \subset   \ \left\{x \in \Omega\colon\ \mathcal{M}(|D u|^{2}) > 1\right\} \cup
\left\{x \in \Omega\colon\ \mathcal{M}(|F|^{2})  >   \delta^{2}\right\}.
\end{equation}
\end{lem}
\begin{proof}
We argue by contradiction.
Assume that $C_{r}(y)$ satisfies \eqref{410} but the claim
\eqref{411} is false. Then there exists a point $y_1 \in
\Omega_{r}(y) = C_{r}(y) \cap \Omega$ such that for each $\rho
> 0$ one has
\begin{equation}\label{412}
\Xint-_{C_{\rho}(y_{1})  } |D u|^2  \; dx \leq 1\quad \textrm{and}\quad \Xint-_{ C_{ \rho }(y_{1})
}|F|^2  \; dx \leq \delta ^{2}.
\end{equation}

We first investigate the interior case when $C_{ 5\sqrt{2} r }(y)  \subset  \Omega.$
According to Definition~\ref{def100}, there exists a new coordinate system, modulo reorientation of the axes and translation, depending on $y$ and $r,$ whose variables we denote by $z,$ such that in this new coordinate system $y=0,$ $y_1=z_1,$
\begin{equation}\label{413}
C_{   4 \sqrt{2} r  }    \subset  C_{  5 \sqrt{2} r } (z_1)   \subset \Omega,
\end{equation}
and
$$
\Xint-_{C_{  4  \sqrt{ 2 }  r } } \left |A(z^\prime, z_n)-
\overline { A} _{B^\prime_{ 4 \sqrt{2} r }} (z_n) \right| ^{2} \; dz \leq \delta^2.
$$
From \eqref{412} and \eqref{413} we see that
$$
\Xint-_{C_{4\sqrt{2} r} } |  D u|^2  \; dz \leq \left(\frac{5}{4}\right)^{n}\quad   \textrm{and}\quad  \Xint-_{C_{ 4 \sqrt{2}r }
}|F|^2   \; dz \leq  \left(\frac{5}{4}\right)^{n} \delta ^{2}.
$$
Applying Lemma~\ref{lem408} with $\rho = \sqrt{2}r$ and $\lambda= \sqrt{\left(\frac{5}{4}\right)^{n} },$
the last three inequalities imply that we are under the hypotheses of Lemma~\ref{lem311} which gives, after scaling back, that for the weak solution $v$ of
\begin{equation}\label{416}
\begin{cases}
D_{i} \left( \overline{a}^{ij}_{B^\prime_{4\sqrt{2}r}} (z_n) D_{j} v\right)
 = &\!\!\! 0  \quad \textrm{in} \ C_{2\sqrt{2}r},\\
\hfill v  = &\!\!\! u \quad \textrm{on} \ \partial C_{2\sqrt{2}r}
\end{cases}
\end{equation}
one has
\begin{equation}\label{417}
\Xint-_{C_{\sqrt{2}r}} |D(u-v)|^2 \;  dz \leq  \varepsilon ^2
\end{equation}
and
\begin{equation}\label{418}
\Vert D v \Vert_{L^\infty(C_{\sqrt{2}r})} \leq \lambda_0
\end{equation}
for some positive constant $\lambda_0=\lambda_0(\nu,L,n).$

We have
\begin{align*}
                \Big| \big\{z \in \Omega\colon\ & \ \mathcal{M}(|D
u|^{2}) >   8 \lambda_0^2\big\}  \cap  C_{r} \Big|\\
                \stackrel{ \eqref{416}}{\leq } & \  \left| \left\{z \in \Omega\colon\ \mathcal{M}_{ C_{2 \sqrt{2}r}} (|D(u-v)|^{2}) + \mathcal{M} _{ C_{2\sqrt{2}r}} (|Dv|^2) >  2 \lambda_0^2 \right\}   \cap  C_{ \sqrt{2}r}  \right|  \\
                  \leq\ \,  & \ \left| \left\{z \in \Omega\colon\ \mathcal{M} _{C_{2\sqrt{2}r}} (|D(u-v)|^{2}) > \lambda_0^2 \right\}    \cap
C_{ \sqrt{2}r}  \right| \\
           & \ \quad       +   \left| \left\{z \in \Omega\colon\ \mathcal{M}_{C_{2\sqrt{2}r}} (|Dv|^{2}) >  \lambda_0^2 \right\}   \cap   C_{ \sqrt{2}r} \right| \\
                  \stackrel{\eqref{418}}{=} &\ \left| \left\{z \in \Omega\colon\ \mathcal{M} _{C_{2\sqrt{2}r}} (|D(u-v)|^{2}) > \lambda_0^2 \right\}    \cap
C_{ \sqrt{2}r}  \right|\\
                \stackrel{\eqref{403-1}}{<}\, &\ \, c     \int_{C_{2 \sqrt{2}r }  } |D (u-v)|^2  \;  dz\\
 \stackrel{\eqref{417}}{\leq}  &   \  \,   c_0      \varepsilon^2 |C_r|,
\end{align*}
for some positive constant $c_0=c_0(\nu, L,n).$ In other words,
$$
\big| \left\{x \in \Omega\colon\ \mathcal{M}(|D
u|^{2}) >   8 \lambda_0^2\right\}  \cap  C_{r}(y) \big|  <  c_0  \varepsilon^2  |C_r(y)|,
$$
whence Lemma~\ref{lem010} yields
\begin{align}\label{419}
w \big( \left\{x \in C_r(y)\colon\  \mathcal{M}(|D
u|^{2})    >       8 \lambda_0^2\right\}\big)  < \ & \gamma_1 \left(c_0 \varepsilon^2\right)^\beta  w\left(C_r(y)\right) \\
\nonumber
  \leq \  &  c_1 \varepsilon^{2 \beta}w\left(C_r(y)\right)
\end{align}
with a positive constant $c_1=c_1\left(\nu,L,n, [w]_{\frac{p}{2}}\right).$

We next consider the boundary case when $C_{  5\sqrt{2} r  }(y)  \not \subset  \Omega.$ Now, for the sake of simplicity, we denote $c$ to mean a universal constant $c=c(\nu,L,n)$ that is independent of $\delta.$ We may also suppose that there is a boundary point $y_{0}\in \partial \Omega \cap
C_{  5\sqrt{2} r  }(y).$  According to Definition~\ref{def100}, there exists a new coordinate system, modulo reorientation of the spatial axes and translation, depending on $y_0$ and $r,$ whose variables we denote by $z,$ such that in this new coordinate system the origin is $y_0 + \delta_0 \overrightarrow{n_0}$ for some small $\delta_0>0$ and some inward unit normal $\overrightarrow{n_0}$ to $\partial \Omega$ at $y_0,$ $y=z_0,$ $y_1=z_1,$
\begin{equation}
\label{420}
C^+_{  24  \sqrt{2} r }   \    \subset   \
\Omega_{   24 \sqrt{2}   r   }    \   \subset       \   \left\{z \in C_{   24\sqrt{2}  r }\colon\  z_n
> -48 \sqrt{2} r \delta\right\},
\end{equation}
and
\begin{equation}
\label{421}
\Xint-_{C^+_{ 24\sqrt{2} r  } } \left |A(z^\prime, z_n) - \overline
{A} _{B^\prime_{24 \sqrt{2} r }} (z_n) \right| ^{2} \; dz \leq \delta^2.
\end{equation}
Then it follows from \eqref{412}, \eqref{420} and \eqref{421} that
\begin{equation}
\label{422}
\Xint-_{\Omega_{24\sqrt{2} r}  } |  Du|^2  \; dz \leq  2 \left(\frac{25}{24}\right)^{n}  \Xint-_{\Omega_{ 25 \sqrt{2} r}(z_1)    } |  D u|^2  \; dz \leq 2 \left(\frac{25}{24}\right)^{n}
\end{equation}
and
\begin{equation}
\label{423}
\Xint-_{\Omega_{24\sqrt{2} r} } | F|^2   \; dz \leq 2 \left(\frac{25}{24}\right)^{n}\delta ^{2}.
\end{equation}
Applying Lemma~\ref{lem408} with $\rho=6\sqrt{2}r$ and $\lambda=\sqrt{ 2 \left(\frac{25}{24}\right)^{n} },$  \eqref{420}--\eqref{423} show
that we are under the hypotheses of Lemma~\ref{lem334}, which gives after backscaling that
there exists a weak solution $v$ of
\begin{equation}
\label{424}
\begin{cases}
\displaystyle  D_i \left( \widetilde{a}^{ij}(z_n) D_{j}v \right)
 = &\!\!\! 0  \quad \textrm{in} \ C^+_{12\sqrt{2}r},\\
\hfill v  = &\!\!\! 0 \quad \textrm{on} \ T_{12\sqrt{2}r},
\end{cases}
\end{equation}
such that
\begin{equation}
\label{425}
\Xint-_{\Omega_{6\sqrt{2}r}} |D(u-v_0)|^2 \  dz \leq  \varepsilon^2
\end{equation}
and
\begin{equation}
\label{426}
\Vert Dv_0 \Vert_{L^\infty\left(C_{6\sqrt{2}r}\right)}=\Vert Dv \Vert_{L^\infty\left(C^+_{6\sqrt{2}r}\right)}  \leq \lambda_1,
\end{equation}
where $v_0$ is the zero extension of $v$ from $C^+_{12\sqrt{2}r}$ to $C_{12\sqrt{2}r}.$

\noindent
We have
\begin{align*}
\big| \{ z \in \Omega\colon \ \mathcal{M}& (|Du|^{2}) >  8 \lambda_1^2\}  \cap  C_{r}(z_0) \big| \\
\stackrel{\eqref{420}, \eqref{424}}{\leq}  & \ \left| \left\{ z \in \Omega_{6\sqrt{2}r}\colon\  \mathcal{M}_{  \Omega_{12\sqrt{2}r}  } (|D(u-v_0)|^{2})+\mathcal{M}_{  \Omega_{12\sqrt{2}r}  } (|Dv_0|^{2}) > 2 \lambda_1^2 \right\}  \right| \\
\stackrel{\eqref{424}}{\leq}\quad \  &\  \left| \left\{ z \in \Omega\colon\  \mathcal{M}_{\Omega_{12\sqrt{2}r}  } (|D(u-v_0)|^{2}) >  \lambda_1^2 \right\}   \cap  C_{6\sqrt{2}r} \right| \\
&\ \quad +  \left| \left\{ z \in \Omega\colon\  \mathcal{M}_{  \Omega_{12\sqrt{2}r}  }  (|Dv_0|^{2}) >  \lambda_1^2 \right\}   \cap  C_{6\sqrt{2}r} \right| \\
\stackrel{\eqref{426}} {=} \quad\ &\  \left| \left\{ z \in \Omega\colon\  \mathcal{M}_{\Omega_{12\sqrt{2}r}  } (|D(u-v_0)|^{2}) >  \lambda_1^2 \right\}   \cap  C_{6\sqrt{2}r} \right| \\
\stackrel{\eqref{403-1}}{  \leq  } \quad\   & \  c   \     \int_{\Omega_{12 \sqrt{2}r }  } |D (u-v_0)|^2  \  dz \\
\stackrel{\eqref{425}} {\leq}   \quad\  & \   c   \   \varepsilon^2 \  |\Omega_{12 \sqrt{2} r}|   <    c  \  \varepsilon |C_r(z_0)|.
\end{align*}
This way
$$
\big| \left\{x \in \Omega\colon\  \mathcal{M}(|D
u|^{2}) >   8 \lambda_1^2\right\}  \cap  C_{r}(y) \big|  <  c  \varepsilon  |C_r(y)|
$$
and therefore Lemma~\ref{lem010} implies
\begin{equation}\label{427}
w\big( \left\{x \in \Omega\colon\  \mathcal{M}(|D
u|^{2}) >   8 \lambda_1^2\right\}  \cap  C_{r}(y) \big)  <  c_2  \varepsilon^{\beta}  w\left(C_r(y)\right)
\end{equation}
with a constant $c_2=c_2\left(\nu,L,n, [w]_{\frac{p}{2}}\right)>0.$

Finally, we combine \eqref{419} and \eqref{427} and set $\lambda_2=\max\{\lambda_0, \lambda_1\}$ in order to reach a contradiction with \eqref{410} since $\varepsilon$ is arbitrary given.
This completes the proof.
\end{proof}

Fix now $\varepsilon$ and take $\delta$ and
$\lambda_2$ as given in Lemma~\ref{lem409}. We use the Vitali covering lemma (Lemma~\ref{lem407}) in order to obtain the power decay of
$$
w\left(\{(x,t) \in \Omega\colon\ \mathcal{M}(|Du|^2)
> \lambda_2^2\}\right).
$$

\begin{lem}\label{lem428}
Let $|F|^2 \in L^{\frac{p}{2}}_w(\Omega) \in L^1(\Omega)$ with $w \in A_\frac{p}{2},$ $2<p< \infty,$ and let $u \in H^1_0(\Omega)$ be the weak solution of \eqref{000}.
Suppose $(A,  \Omega)$ is $(\delta, R)$-vanishing of codimension $1$ and
set $\varepsilon_1=\gamma_1 \varepsilon.$ Then we have
\begin{align}\label{429}
&    w\big(\{x \in \Omega\colon\  \mathcal{M}(|D u|^{2}) > \lambda_2^{2k} \} \big)
\leq
\varepsilon _{1}^{k} w\big( \{x \in \Omega\colon\   \mathcal{M}(|D u|^{2}) >  1  \}
\big)\\
\nonumber
  &  \qquad\qquad  +  \sum ^{k}_{i=1}     \varepsilon _{1}^{i}
w\left( \{x\in \Omega\colon\  \mathcal{M} (|F|^{2}) >   \delta ^{2}
\lambda_2^{2( k-i )} \} \right)  \quad (k=1,2, \cdots).
\end{align}
\end{lem}
\begin{proof}
We set
$$
D=\{x \in \Omega\colon\  \mathcal{M}(|D u|^{2})
> \lambda_2^2\}
$$
and
$$
E=\{x \in \Omega\colon\  \mathcal{M}(|D u|^{2})
> 1\} \cup \{x \in \Omega\colon\  \mathcal{M}(|F|^{2}) >
\delta^{2}\}.
$$
Clearly, $D\subset E \subset \Omega$ and $w \left(D \cap C_1(y)
\right)< \varepsilon w \left(C_1(y)\right)$ by Lemma \ref{lem010}
and a proper normalization of the problem \eqref{000}, as one can
take $\lambda$ large enough in Lemma~\ref{lem408}. The second
condition of Lemma~\ref{lem407} is a direct consequence of
Lemma~\ref{lem409}. Applying Lemma~\ref{lem407}, we get the claim in
the case $k=1$ and iteration of the foregoing arguments leads to
\eqref{429} for each $k>1.$
\end{proof}

We are now ready to give a complete proof of our main result, Theorem~\ref{thm104}. For, we will employ Lemma~\ref{lem404} with $s= \frac{p}{2},$ $h=\mathcal{M}(|Du|^2),$ $\lambda=\lambda^2_2,$ $\theta=1.$
We have
\begin{align*}
            \sum ^{  \infty  }_{k=1} \lambda_2^{ 2(k \frac{p}{2})}& w\left(\left\{\mathcal{M}( |D u|^{2} )
>  \lambda_2^{2k}\right\}\right)\\
 \stackrel{ \eqref{429}}{\leq } \ \,\,\, &\ \sum ^{  \infty  }_{  k=1} \lambda_2^{ 2(k \frac{p}{2})} \varepsilon _{1}^{k} w\left( \{x \in \Omega\colon\   \mathcal{M}(|D u|^{2}) >  1  \}
\right) \\
&\quad +    \sum ^{  \infty  }_{  k=1} \lambda_2^{ 2(k \frac{p}{2})} \varepsilon _{1}^{k} \sum ^{k}_{i=1}     \varepsilon _{1}^{i}
w\left( \{x\in \Omega\colon\  \mathcal{M} (|F|^{2}) >   \delta ^{2}
\lambda_2^{2( k-i )} \} \right)\\
  \leq\quad\ \, &\  \sum ^{  \infty  }_{  k=1}   \left(\lambda_2 ^{ p} \varepsilon _{1}\right)^{k} w\left( \Omega\right) \\
&\quad +    \sum ^{  \infty  }_{  i=1 }     \left(\lambda_2 ^{ p} \varepsilon _{1} \right)^{i} \left( \sum ^{ \infty }_{ k=i }  \lambda_2 ^{ 2 (k-i)\frac{p}{2} }
w\left( \left\{x\in \Omega\colon\  \mathcal{M} (|F|^{2}) >    \delta ^{2}
\lambda_2^{2( k-i )}\right\}\right) \right) \\
  \stackrel{\textrm{Lemma~\ref{lem404}}}{\leq} &\ c \left(w(\Omega) + \Vert \mathcal{M} (|F|^{2})\Vert^{\frac{p}{2}}_{ L^{ \frac{p}{2}}_{w}(\Omega)} \right)  \sum ^{  \infty  }_{  k=1}   \left(\lambda_2 ^{ p} \varepsilon _{1}\right)^{k}\\
 \stackrel{\textrm{Lemma~\ref{lem401}}}{\leq} &\  c \left(w(\Omega) + \Vert |F|^{2}\Vert^{\frac{p}{2}}_{ L^{ \frac{p}{2}}_{w}(\Omega)} \right) \sum ^{  \infty  }_{  k=1}   \left(\lambda_2 ^{ p} \varepsilon _{1}\right)^{k}
\end{align*}
for some universal constant $c=c\left(\delta, \nu,L,n,[w]_{\frac{p}{2}}\right)>0.$

We then first select $\varepsilon>0$ so small to have $\lambda_2 ^{ p} \varepsilon _{1}<1.$
In view of Lemma~\ref{lem409}, one can find a small constant $\delta=\delta \left(\nu, L, n, [w]_{\frac{p}{2}}\right)$  such that
$$
\sum ^{  \infty  }_{  k=1} \lambda_2^{ 2(k \frac{p}{2})} w\left(\left\{\mathcal{M}( |D u|^{2} )
>  \lambda_2^{2k}\right\}\right) \leq  c \left(w(\Omega) + \Vert |F|^{2}\Vert^{\frac{p}{2}}_{ L^{ \frac{p}{2}}_{w}(\Omega)} \right)
$$
holds true for now fixed small $\delta>0$ for all $(A,\Omega)$ which are $(\delta,R)$-vanishing of codimension $1.$
Therefore, it follows from Lemmae~\ref{lem401} and \ref{lem404} that
$$
\int_{\Omega} |Du|^{p} w(x) \; dx \leq c \left(\int_{\Omega} |F|^{p} w(x) \; dx
+ w(\Omega)\right),
$$
which implies the desired estimate \eqref{105} through the Banach inverse
mapping theorem. This completes the proof of Theorem~\ref{thm104}.\hfill\qed

\section{Morrey regularity of the weak solution}\label{sec6}
\setcounter{equation}{0}
\setcounter{thm}{0}

We will apply now our main result Theorem~\ref{thm104} in obtaining gradient estimates in Morrey spaces for the weak solutions to \eqref{000}.

Let us start, first of all, with recalling the definition of the Morrey spaces. Given a bounded domain $\Omega\subset\mathbb{R}^n,$ $p\in(1,\infty)$ and $\lambda\in(0,n),$ a function $f\in L^p(\Omega)$ is said to belong to the Morrey space $L^{p,\lambda}(\Omega)$ if
$$
\|f\|_{L^{p,\lambda}(\Omega)}=\sup_{x_0\in\Omega,\, r\in(0,\mathrm{diam\,}\Omega)}
\left( \frac{1}{r^{\lambda}} \int_{B_r(x_0)\cap\Omega} |f(x)|^p\; dx\right)^{1/p}
<\infty,
$$
where $B_r(x_0)$ is a ball centered at $x_0$ and of radius $r>0.$  The above quantity defines a norm under which $L^{p,\lambda}(\Omega)$ becomes a Banach space. The limit cases $\lambda=0$ and $\lambda=n$ give rise to $L^p(\Omega)$ and $L^\infty(\Omega),$ respectively.

The following result extends the $W^{1,p}(\Omega)$-regularity theory of
\eqref{000} to the settings of Morrey spaces.
\begin{thm}\label{thm6}
Assume \eqref{001} and \eqref{002}. Given $p \in (2, \infty)$ and  $\lambda\in(0,n),$ there exists a small
positive constant $\delta$  and a positive constant $c,$ depending on $n,$ $L,$ $\nu,$ $p,$ $\lambda$ and $\Omega,$ such that if $(A, \Omega)$ is
($\delta,R)$-vanishing of codimension $1$ and $F\in L^{p,\lambda}(\Omega,\mathbb{R}^n),$ then the unique
weak solution $u \in H^1_0(\Omega)$ of the Dirichlet problem \eqref{000} satisfies $Du\in L^{p,\lambda}(\Omega,\mathbb{R}^n)$ with the estimate
\begin{equation}\label{601}
\|Du\|_{L^{p,\lambda}(\Omega,\mathbb{R}^n)}
\leq c \|F\|_{L^{p,\lambda}(\Omega,\mathbb{R}^n)}.
\end{equation}
\end{thm}
\begin{proof}
Extend $F$ as zero outside $\Omega$ and fix arbitrary $x_0\in\Omega$ and $r>0.$ Set $\chi_{B_r(x_0)}$ for the characteristic function of the ball $B_r(x_0)$ and $\mathcal{M}\chi_{B_r(x_0)}(x)$ for its Hardy--Littlewood maximal function.

It follows from Proposition~2 in \cite{CR} that if $\sigma\in(0,1)$ then
$\left(\mathcal{M}\chi_{B_r(x_0)}(x)\right)^\sigma$ belongs to the Muckenhoupt class $A_1,$ that is,
$$
\mathcal{M}\Big(\left(\mathcal{M}\chi_{B_r(x_0)}(x)\right)^\sigma\Big) \leq
c \left(\mathcal{M}\chi_{B_r(x_0)}(x)\right)^\sigma\quad \textrm{for a.a.}\ x\in \mathbb{R}^n,
$$
whence $\left(\mathcal{M}\chi_{B_r(x_0)}(x)\right)^\sigma\in A_{\frac{p}{2}}$ for each $p\in(2,\infty)$ with
$$
\left[\left(\mathcal{M}\chi_{B_r(x_0)}(x)\right)^\sigma\right]_{\frac{p}{2}}=c(n,p,\sigma).
$$

Choosing $\sigma\in\left(\frac{\lambda}{n},1\right),$ it follows from Theorem~\ref{thm104} that there exist constants $\delta$ and $c,$ depending on $n,$ $L,$ $\nu,$ $p,$ $\lambda$ and $\Omega,$ such that if $(A,\Omega)$ is $(\delta,R)$-vanishing of codimension $1,$ then
\begin{align}\label{602}
\int_{B_r(x_0)\cap\Omega} |Du&(x)|^p\; dx =  \int_\Omega |Du(x)|^p \left(\chi_{B_r(x_0)}(x)\right)^\sigma \; dx\\
\nonumber
 \leq\ \, & \ \int_\Omega |Du(x)|^p \left(\mathcal{M}\chi_{B_r(x_0)}(x)\right)^\sigma \; dx\\
\nonumber
 \stackrel{ \eqref{105}}{\leq } & \ c\int_\Omega |F(x)|^p \left(\mathcal{M}\chi_{B_r(x_0)}(x)\right)^\sigma \; dx\\
\nonumber
 = \ \, & \ c\int_{\mathbb{R}^n} |F(x)|^p \left(\mathcal{M}\chi_{B_r(x_0)}(x)\right)^\sigma \; dx\\
\nonumber
= \ \, & \ c\left( \underbrace{\int_{B_{2r}(x_0)} |F(x)|^p \left(\mathcal{M}\chi_{B_r(x_0)}(x)\right)^\sigma \; dx}_{I_0(r,x_0)}\right.\\
\nonumber
 & \qquad +\left.\sum_{k=1}^\infty \underbrace{\int_{B_{2^{k+1}r}(x_0)\setminus B_{2^{k}r}(x_0)} |F(x)|^p \left(\mathcal{M}\chi_{B_r(x_0)}(x)\right)^\sigma \; dx}_{I_k(r,x_0)}\right)
\end{align}
after the dyadic decomposition $\mathbb{R}^n=B_{2r}(x_0)\bigcup\Big(\bigcup_{k=1}^\infty B_{2^{k+1}r}(x_0)\setminus B_{2^{k}r}(x_0)\Big).$

We have $\mathcal{M}\chi_{B_r(x_0)}(x)\leq 1$ a.e. $\mathbb{R}^n$ whence
\begin{equation}\label{603}
I_0(r,x_0)\leq \int_{B_{2r}(x_0)} |F(x)|^p \; dx\leq c(n) r^\lambda \|F\|_{L^{p,\lambda}(\Omega,\mathbb{R}^n)}^p.
\end{equation}

To estimate $I_k(r,x_0),$ we note that for each $x\in B_{2^{k+1}r}(x_0)\setminus B_{2^{k}r}(x_0)$ and each $\rho>0$ we have
\begin{equation}\label{604}
\frac{1}{|B_\rho(x)|} \int_{B_\rho(x)} \left|\chi_{B_r(x_0)}(x)\right|\; dx
\leq \frac{|B_r(x_0)|}{|B_\rho(x)|}=\frac{r^n}{\rho^n}.
\end{equation}
Moreover, having in mind $x\in B_{2^{k+1}r}(x_0)\setminus B_{2^{k}r}(x_0),$ the term on the left-hand side above is positive only for values of $\rho$ greater than $2^kr-r.$ This way, the obvious inequality $2^k-1\geq 2^{k-1}$ $\forall k\geq 1$ reduces \eqref{604} to
$$
\frac{1}{|B_\rho(x)|} \int_{B_\rho(x)} \left|\chi_{B_r(x_0)}\right|(x)\; dx
\leq \frac{r^n}{2^{n(k-1)}r^n}=\frac{1}{2^{n(k-1)}}
$$
and taking supremum with respect to $\rho>0$ we obtain
$$
\left(\mathcal{M}\chi_{B_r(x_0)}(x)\right)^\sigma \leq
\frac{1}{2^{\sigma n(k-1)}}.
$$

Therefore,
\begin{align}\label{605}
I_k(r,x_0) \leq&\ \frac{1}{2^{\sigma n(k-1)}} \int_{B_{2^{k+1}r}(x_0)\setminus B_{2^{k}r}(x_0)} |F(x)|^p \; dx\\
\nonumber
\leq &\ \frac{1}{2^{\sigma n(k-1)}} \int_{B_{2^{k+1}r}(x_0)} |F(x)|^p \; dx\\
\nonumber
\leq &\ \frac{(2^{k+1}r)^\lambda}{2^{\sigma n(k-1)}} \frac{1}{(2^{k+1}r)^\lambda} \int_{B_{2^{k+1}r}(x_0)} |F(x)|^p \; dx\\
\nonumber
\leq &\ 2^{\lambda+\sigma n} (2^{\lambda-\sigma n})^k r^\lambda
\|F\|_{L^{p,\lambda}(\Omega,\mathbb{R}^n)}^p.
\end{align}

Substitution of \eqref{603} and \eqref{605} into \eqref{602} yields
$$
\int_{B_r(x_0)\cap\Omega} |Du(x)|^p\; dx\leq c r^\lambda \left(
\sum_{k=0}^\infty (2^{\lambda-\sigma n})^k\right)
\|F\|_{L^{p,\lambda}(\Omega,\mathbb{R}^n)}^p=c r^\lambda \|F\|_{L^{p,\lambda}(\Omega,\mathbb{R}^n)}^p
$$
with a convergent series thanks to the choice $\sigma\in\left(\frac{\lambda}{n},1\right).$ Dividing the both sides of the last inequality by $r^\lambda$ and taking the supremum with respect to $x_0\in\Omega$ and $r>0$ gives the desired estimate  \eqref{601}.
\end{proof}

The gradient estimate \eqref{601} and the known properties of functions with Morrey regular gradient (see Lemmae~3.III and 3.IV in \cite{Cm}) imply immediately better integrability and H\"{o}lder continuity of the weak solution to \eqref{000} for appropriate values of $p$ and $\lambda.$
Namely,
\begin{crlr}
Under the assumptions of Theorem~\ref{thm6} let $u\in H^1_0(\Omega)$ be the weak solution of \eqref{000}. Then
\begin{enumerate}
\item $u\in L^{\frac{np}{n-p},\frac{n\lambda}{n-p}}(\Omega)\subset L^{p,\lambda+p}(\Omega)$ if $p+\lambda<n;$
\item $u\in L^{p',\lambda'}(\Omega)$ for any $p'<\infty$ and any $\lambda'<n,$ if $p+\lambda=n;$
\item $u\in C^{0,1-\frac{n-\lambda}{p}}(\overline\Omega)$ if $p+\lambda>n.$
\end{enumerate}
\end{crlr}

It is worth noting that the global H\"{o}lder continuity with \textit{some} exponent for the weak solutions to divergence form elliptic equations with \textit{only} measurable coefficients is the essence of the celebrated De~Giorgi regularity result (\cite{DG}) when $F\in L^p(\Omega,\mathbb{R}^n)$ with $p>n,$ and that of Morrey \cite{Mo0} when $F\in L^{p,\lambda}(\Omega,\mathbb{R}^n)$ with $p+\lambda>n,$ both holding in domains with H\"{o}lder continuous boundaries. Apart from the fact that we are dealing with Reifenberg flat domains, in our more restricted situation (coefficients which are measurable in one variable and small BMO in the remaining ones) we provide
 an explicit expression for the H\"{o}lder exponent of the weak solution.

The results from this section will be applied in a forthcoming paper to the study of Morrey regularity of weak solutions to quasilinear divergence form elliptic equations with controlled growths of the nonlinearities.

\bibliographystyle{amsplain}

\end{document}